\documentclass[noinfoline]{article}

\RequirePackage[OT1]{fontenc}
\RequirePackage[
amsthm,amsmath
]{imsart}

\usepackage{graphicx}
\usepackage{latexsym,amsmath}
\usepackage{amsmath,amsthm,amscd}
\usepackage{amsfonts}
\usepackage[psamsfonts]{amssymb}
\usepackage{enumerate}
\usepackage{cite}

\usepackage{url}
\usepackage{tocvsec2}

\usepackage{epstopdf}
\usepackage{float}

\startlocaldefs

\setattribute{keyword}{AMS}{AMS 2010 subject classifications:}

\theoremstyle{plain} 
\newtheorem{theorem}{Theorem}
\newtheorem{corollary}[theorem]{Corollary}

\newtheorem{lemma}
{Lemma}

\theoremstyle{definition} 

\theoremstyle{definition} 
\newtheorem{ex}[theorem]{Example}
\newtheorem*{ex*}{Example}
\theoremstyle{remark} 

\theoremstyle{remark} 

\newtheorem*{remark*}{Remark}

\newcommand{\dd}{\partial}

\renewcommand{\dd}{{\operatorname{d}}}

\renewcommand{\a}{k}

\newcommand{\al}{\alpha}

\newcommand{\si}{\sigma}

\newcommand{\la}{\lambda}

\newcommand{\be}{\beta}

\renewcommand{\th}{\theta}

\newcommand{\intr}[2]{\overline{#1,#2}}

\renewcommand{\P}{\operatorname{\mathsf{P}}} 
\newcommand{\E}{\operatorname{\mathsf{E}}}

\newcommand{\Z}{\mathbb{Z}}
\newcommand{\R}{\mathbb{R}}

\newcommand{\F}{\mathcal{F}}

\newcommand{\J}{\mathcal{J}}

\newcommand{\A}{\mathcal{A}}

\newcommand{\tc}{{\tilde{c}}}

\renewcommand{\le}{\leqslant}
\renewcommand{\ge}{\geqslant}

\renewcommand{\a}[2]{a_{#1}^{(#2)}}
\renewcommand{\A}[2]{A_{#1}^{(#2)}}
\renewcommand{\c}[2]{c_{#1}^{(#2)}}
\renewcommand{\tc}[2]{\tilde{c}_{#1}^{(#2)}}

\newcommand{\X}{{\mathfrak{X}}}
\newcommand{\W}{{\mathfrak{P}}}

\endlocaldefs

\begin{document}


\begin{frontmatter}

\title{Rosenthal-type inequalities for martingales in 2
-smooth Banach spaces}
\runtitle{Martingales in $2$-smooth spaces}

%

\begin{aug}
\author{\fnms{Iosif} \snm{Pinelis}\thanksref{t2}\ead[label=e1]{ipinelis@mtu.edu}
}
  \thankstext{t2}{Supported by NSA grant H98230-12-1-0237}
\runauthor{Iosif Pinelis}


\address{Department of Mathematical Sciences\\
Michigan Technological University\\
Houghton, Michigan 49931, USA\\
E-mail: \printead[ipinelis@mtu.edu]{e1}}
\end{aug}

\begin{abstract} 
Certain previously known upper bounds on the moments of the norm of martingales in 2-smooth Banach spaces are improved. Some of these improvements hold even for sums of independent real-valued random variables. Applications   to concentration of measure on product spaces for
separately Lipschitz functions are presented, including ones concerning the central moments of the norm of the sums of independent random vectors in any separable Banach space.     
\end{abstract}

  
%





\end{frontmatter}

\settocdepth{chapter}


\settocdepth{subsubsection}

\theoremstyle{plain} 


\section{Introduction, summary, and discussion}\label{intro} 
Let $(S_i)_{i=0}^\infty$ be a martingale in a separable Banach space $(\X,\|\cdot\|)$ relative to a filter $(\F_i)_{i=0}^\infty$ of $\si$-algebras. 
Assume that $S_0=0$ and the differences $X_i:=S_i-S_{i-1}$ 
satisfy the condition 
\begin{equation}\label{eq:cond}
	\E_{i-1}\|X_i\|^2\le b_i^2
\end{equation}
almost surely for all $i\in\intr1\infty$, where 
$\E_{i-1}$ denotes the conditional expectation given $\F_{i-1}$ and 
the $b_i$'s are some positive real numbers.   
Here and in what follows, for any $m$ and $n$ in $\Z\cup\{\infty\}$ we let $\intr mn:=\{i\in\Z\colon m\le i\le n\}$. 
Also introduce  
\begin{equation*}
	\a it:=\E\|X_i\|^t, \quad \A nt:=\sum\limits_{i\in\intr1n} a_i(t), \quad\text{and}\quad B_n:=\sqrt{\sum_{i\in\intr1n} b_i^2} 
\end{equation*}
for all real $t\ge0$, $i\in\intr1\infty$, and $n\in\intr0\infty$; as usual, assume that the sum and product of an empty family are $0$ and $1$, respectively; let also $0^0:=1$. 

Assume further that the Banach space $(\X,\|\cdot\|)$ is 2-smooth \big(or, more exactly, $(2,D)$-smooth, for some $D = D(\X ) >0$\big), in the following sense: 
\begin{equation}\label{eq:2-smooth}
	\| x+y\| ^2 + \| x-y\| ^2 \le 2\| x\| ^2 +2D^2\| y\| ^2  
\end{equation}
for all $x$ and $y$ in $\X$. 
The importance  of the 2-smooth spaces was elucidated in 
\cite{pisier75}: they play the same role with respect to the vector martingales  as the spaces of type 2 do with respect to the sums of independent random vectors.
The definition of 2-smooth spaces assumed in this paper is the same as that in \cite{pin94}, which is slightly different from the one given 
in \cite{pisier75} -- which latter required only that \eqref{eq:2-smooth} hold for an equivalent norm; the reason for the modified definition is that 
we would like to follow the dependence of certain constants on $D$, the constant of 2-smoothness.
Substituting $\lambda x$ for $y$ in \eqref{eq:2-smooth}, where $\lambda \in \R$, one observes that necessarily $D\ge 1$, except for $\X=\{0\}$. 
As shown in \cite{pin94,wellner_nemir}, the space $L^p (T, {\cal A}, \nu)$ is $(2,\sqrt {p-1} \, )$-smooth for any $p\ge 2$ and any measure space
$(T,{\cal A}, \nu)$. 
In particular, what is obvious and well-known, if $\X $ is a Hilbert space then it is (2,1)-smooth.  

\begin{theorem}\label{th:}
Take any real $t\ge0$ and $n\in\intr0\infty$. Then  
\begin{equation}\label{eq:}
	\begin{aligned}
	\E\|S_n\|^t\le
	\sum_{j\in\intr0{m-1}}\c jt & \sum_{J\in\J_{n,j}} \A{\mu_n(J)-1}{t-2j} \prod_{i\in J}b_i^2 \\ 
	+\tc mt & \sum_{J\in\J_{n,m}} \big(\A{\mu_n(J)-1}2\big)^{t/2-m} \prod_{i\in J}b_i^2, 
	\end{aligned}
\end{equation}
where $m:=\lfloor\frac t2\rfloor$, 
$\J_{n,j}$ 
denotes the set of all subsets of the set $\intr1n$ of cardinality $j$, 
$\mu_n(J):=\min J$ if $J\ne\emptyset$, $\mu_n(\emptyset):=n+1$, 
\begin{align*}
	\c jt:=
	\frac{t-2j-2+D^2}{t-2j-1}\,q(t-2j)
	\prod_{k\in\intr0{j-1}}\frac{(t-2k)(t-2k-2+D^2)\,p(t-2k)}2, \\
	\tc mt:=
	\prod_{j\in\intr0{m-1}}\frac{(t-2j)(t-2j-2+D^2)\,p(t-2j)}2, 
\end{align*}
and $p(\cdot)$ and $q(\cdot)$ are any functions such that 
\begin{equation}
\begin{gathered}\label{eq:p(s),q(s)}
	p(2)+q(2)\ge1,\quad p(2)>0,\quad q(2)>0, \\
	p(s)\ge1 \quad\text{and}\quad q(s)\ge1\quad\text{for all}\quad s>2, \\ 
	p(s)^{1/(3-s)}+q(s)^{1/(3-s)}\le1\quad\text{for all}\quad s>3. 
\end{gathered}	
\end{equation}
\end{theorem}

Almost the same result was obtained in \cite{pin80} in the special case when $\X$ is a Hilbert space (in which case one can take $D=1$) --  except that, 
in place of the term $\A{\mu_n(J)-1}2$ in \eqref{eq:} above, the bound in \cite{pin80} contained the larger term $\A n2-\sum_{i\in J}\a i2$. 

The proofs, whenever necessary, are deferred to Section~\ref{proofs}. 

\begin{corollary}\label{cor:}
In the conditions of Theorem~\ref{th:}, let $t>2$ and take any positive real numbers $\la_j$ for all $j\in\intr0{m-1}$. Then 
\begin{equation}\label{eq:cor}
	\E\|S_n\|^t\le C_A^{(t)}\A nt+C_B^{(t)}B_n^t,
\end{equation} 
where 
\begin{equation}\label{eq:CA,CB}
	\begin{aligned}
	C_A^{(t)}&:=\sum_{j\in\intr0{m-1}}\c jt  \,\frac{t-2j-2}{t-2}\frac1{\la_j^{2j}j!}
	\quad\text{and} \\ 
	C_B^{(t)}&:=\tc mt \prod_{j\in\intr1m}\frac1{t/2-m+j} + \,\sum_{j\in\intr0{m-1}}\c jt\,\frac{2j}{t-2}\frac{\la_j^{t-2j-2}}{j!}.  
	\end{aligned}
\end{equation} 
\end{corollary}

In the special case when $\X$ is a Hilbert space, a similar but somewhat less precise bound was obtained in \cite[Corollary]{pin80}. Namely, in place of the terms $\frac{t-2j-2}{t-2}$, $\frac1{t/2-m+j}$, and $\frac{2j}{t-2}$ in \eqref{eq:CA,CB}, the bound in \cite[Corollary]{pin80} contains the larger terms $1$, $\frac1j$, and $1$, respectively. \\

Results somewhat similar to Corollary~\ref{cor:} were also obtained in \cite{pin94}, by rather different methods. Particularly, \cite[Theorem~4.1]{pin94} implies the ``spectrum'' of inequalities  
\begin{equation}\label{eq:pin94}
	\E\|S_n\|^t\le K^t\big(c^t\A nt+c^{t/2}e^{t^2/c}D^tB_n^t\big)
\end{equation}
for some positive absolute constant $K$ and all real $t\ge2$, depending on the freely chosen value of the ``balancing'' parameter $c\in[1,t]$. 
One can use this freedom to minimize or quasi-minimize the bound in \eqref{eq:pin94} in $c\in[1,t]$, depending on the value of the ratio $\A nt/B_n^t$ -- as is done in \cite{pin94}; see Theorem~6.1 and the definition of $B_p^*$ on page~1693 there. 
In fact, \cite[Theorems~6.1 and 6.2]{pin94} show that 
\begin{enumerate}[(i)]
	\item for each $t\ge2$ and each value of the ratio $\A nt/B_n^t$, the minimum in $c\in[1,t]$ of the upper bound in \eqref{eq:pin94} is an optimal (up to a factor of the form $K^t$) upper bound on $\E\|S_n\|^t$; 
	\item the spectrum of the bounds  in \eqref{eq:pin94} is also ``minimal'' in the sense that for each $c\in[1,t]$ there exists a value of the ratio $\A nt/B_n^t$ such that the corresponding ``individual'' bound in \eqref{eq:pin94} is the best possible (again up to a factor of the form $K^t$);
	\item the above statements (i) and (ii) hold if the term(s) $\A nt$ and/or $B_n^t$ are/is replaced, respectively, by $\E\max_{i\in\intr1n}\|X_i\|^t$ and/or $\E\big(\sum_1^n\E_{i-1}\|X_i\|^2
	\big)^{t/2}$. 
\end{enumerate}

One can similarly use the balancing parameters $\la_j$ in \eqref{eq:CA,CB} to minimize or quasi-minimize the bound in \eqref{eq:cor}. 
Moreover, 
by Remark~6.8 in \cite{pin94} (with details given by \cite[Proposition~9.2
]{pin94-arxiv}), the minimum of the bound on $\E\|S_n\|^t$ in \cite[Theorem~1]{pin80} with respect to the corresponding balancing parameters is equivalent (once again up to a factor of the form $K^t$) to the minimum in $c\in[1,t]$ of the bound in \eqref{eq:pin94} -- at least when the condition \eqref{eq:cond} holds; hence, the same is true for the bounds in \eqref{eq:cor} and \eqref{eq:}, which therefore possess the same up-to-the-$K^t$-factor optimality property. 

Thus, the main advantage of inequality \eqref{eq:pin94} over \eqref{eq:cor} and \eqref{eq:} is that the former does not require the condition \eqref{eq:cond}. On the other hand, the bounds in \eqref{eq:cor} and \eqref{eq:} are quite explicit and do not contain unspecified constants such as $K$ in \eqref{eq:pin94}. 
If one follows the lines of the proof of \eqref{eq:pin94} in \cite{pin94} without serious efforts at modification, the resulting value of $K$ turns out to be large, namely equal $120$, quite in contrast with the constant factors in \eqref{eq:cor} and \eqref{eq:} -- cf.\ e.g.\ the bound in \eqref{eq:t=3} below, which is a particular case of \eqref{eq:cor}. 
Moreover, \eqref{eq:cor} and \eqref{eq:} appear to provide a better dependence of the bounds on the 2-smoothness constant $D$.  


As for the condition \eqref{eq:cond}, it turns out quite naturally satisfied in applications to concentration of measure on product spaces for separately Lipschitz functions and, in particular, for the norm of the sums of independent random vectors -- cf.\ e.g.\ \cite{re-center} and further bibliography there. Such an application is provided in Section~\ref{concentr} of the present note.



Special cases of Corollary~\ref{cor:} are the following inequalities:  
\begin{align}
	\E\|S_n\|^t &\le\frac{t-2+D^2}{t-1}\big(\A nt+(t-1)B_n^t\big)\quad\text{for all}\quad t\in(2,3]\quad\text{and} \label{eq:2<t<3} \\ 
	\E\|S_n\|^t &\le\frac{t-2+D^2}{t-1}\Big(\frac{\A nt}{\al^{t-3}}
	+(t-1)\frac{B_n^t}{(1-\al)^{t-3}}\Big)\quad\text{for all}\quad t\in[3,4] \label{eq:3<t<4}
\end{align}
and $\al\in(0,1)$. 
Minimizing the upper bound in \eqref{eq:3<t<4} in $\al\in(0,1)$, one can combine \eqref{eq:2<t<3} and \eqref{eq:3<t<4} into the inequality  
\begin{equation}\label{eq:2<t<4,min}
		\E\|S_n\|^t\le\frac{t-2+D^2}{t-1}\big[\big(\A nt\big)^{1/s_t}
	+(t-1)^{1/s_t}B_n^{t/s_t}\big]^{s_t}\quad\text{for all}\quad t\in(2,4],   
\end{equation}
where $s_t:=\max(1,t-2)$. 


Note that the coefficient of $\A nt$ in \eqref{eq:2<t<3} cannot be less than $1$, 
whatever the coefficient of $B_n^t$ in \eqref{eq:2<t<3} may be; for instance, take $n=1$ and let $X_1$ be such that $\P(X_1=0)=1-2p$ and $\P(X_1=1)=\P(X_1=-1)=p$, with $p\downarrow0$.   
So, the coefficient of $\A nt$ in \eqref{eq:2<t<3} takes the optimal value $1$ when $D=1$. 

Consider the now the case when the latter condition holds, that is, when $\X$ is a Hilbert space. 
Then the behavior of the upper bound in \eqref{eq:2<t<4,min} depends mainly on the ratio $r:=r_n^{(t)}:=\A nt/B_n^t$. 
If $r\to\infty$ \big(which happens e.g.\ in the situation described in the previous paragraph) then the upper bound in \eqref{eq:2<t<4,min} is asymptotic to $\A nt$ and thus is asymptotically optimal. 
If $r\downarrow0$ \big(which happens e.g.\ in the important case when $n\to\infty$, the $X_i$'s are iid, $t$ is fixed, and $\E|X_1|^t<\infty$\big) then the upper bound in \eqref{eq:2<t<4,min} is asymptotic to $(t-1)B_n^t$. 
On the other hand, by the central limit theorem, this upper bound cannot be less than $\E|Z|^t B_n^t$, where $Z$ is a standard normal r.v. Since $\E|Z|^t=t-1$ for $t\in\{2,4\}$, it follows that the upper bound in \eqref{eq:2<t<4,min} is asymptotically optimal when $r\downarrow0$ and at that either $t\downarrow2$ or $t\to4$. 
The graph of the ratio of $t-1$ to $\E|Z|^t$ in Figure~\ref{fig:rat} shows that the asymptotic coefficient $t-1$ at $B_n^t$ in \eqref{eq:2<t<4,min} for $r\downarrow0$ is rather close to optimality for all $t\in(2,4]$. 

\begin{figure}[h]
	\centering	\includegraphics[width=.7\textwidth]{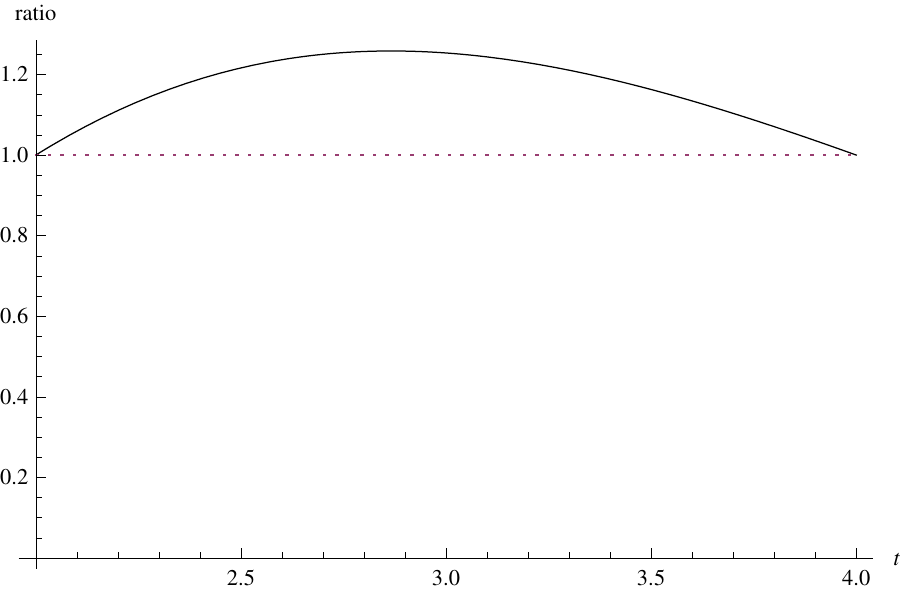}
	\caption{Graph of the ratio of $t-1$ to $\E|Z|^t$ for $t\in(2,4]$. }
	\label{fig:rat}
\end{figure}

%
%

Again when $\X$ is a Hilbert space, \eqref{eq:2<t<3} and \eqref{eq:3<t<4} with $\al=\frac12$ yield  
\begin{equation}\label{eq:2<t<4,H}
	\E\|S_n\|^t\le2^{(t-3)_+}\big(\A nt+(t-1)B_n^t\big)\quad\text{for all}\quad t\in(2,4], 
\end{equation}
where $u_+:=\max(0,u)$. This latter result was obtained, by a more Stein-like method, in \cite[Lemma~6.3]{chen-shao05} in the case when $\X=\R$, $t\in(2,3]$, and the $X_i$'s are independent. 
In the case when $\X=\R$, $t=3$, and the $X_i$'s are iid, inequality \eqref{eq:2<t<4,H} was stated without proof in \cite[page~341]{novak05};  
in the more general case when $\X=\R$, $t\in(2,4]$, and the $X_i$'s are independent but not necessarily identically distributed, the bound 
in \eqref{eq:2<t<4,H} was obtained in \cite[Lemma~13]{novak00}, but with the larger factor $t(t-1)2^{-t/2}$ in place of $t-1$.

In the case $t=3$, which is particularly important in applications to Berry--Esseen-type bounds (see e.g.\ \cite{chen-shao05,nonlinear}), \eqref{eq:2<t<3} or \eqref{eq:3<t<4} 
yields 
\begin{equation}\label{eq:t=3}
	\E\|S_n\|^3\le\tfrac{1+D^2}2\,(\A n3+2B_n^3).
\end{equation}



Rosenthal-type inequalities and related results can be found, among other papers, in \cite{rosenthal,burk,MR0443034,pin-utev84,utev-extr,sibam,ibr-shar97,latala-moments,gine-lat-zinn,
ibr-sankhya,bouch-etal
}. 


\section{Application to concentration of measure on product spaces for separately Lipschitz functions}\label{concentr}

In this section, let us re-define $X_1,\dots,X_n$ to be independent r.v.'s with values in measurable spaces $\X_1,\dots,\X_n$, respectively. Let $g\colon\W\to\R$ be a measurable function on the product space $\W:=\X_1\times\dots\times\X_n$. 
Let us say (cf.\ \cite{bent-isr,normal}) that $g$ is {\em separately Lipschitz} if it satisfies a Lipschitz-type condition in each of its arguments: 
\begin{equation}\label{eq:Lip}
|g(x_1,\dots,x_{i-1},\tilde x_i,x_{i+1},\dots,x_n) -
g(x_1,\dots,x_n)| \le \rho_i(\tilde x_i,x_i)
\end{equation}
for some measurable functions $\rho_i\colon\X_i\times\X_i\to\R$ and 
all $i\in\intr1n$, $(x_1,\dots,x_n)\in\W$, and $\tilde x_i\in\X_i$. 
Take now any separately Lipschitz function $g$ and let  
$$Y:=g(X_1,\dots,X_n).$$
Suppose that the r.v.\ $Y$ has a finite mean. 
Take any real $t>2$. 



\begin{corollary}\label{cor:concentr}
For each $i\in\intr1n$, take any $x_i$ and $y_i$ in $\X_i$.  
Then 
\begin{equation*}
	\E|Y-\E Y|^t
	\le C_t\, C_A^{(t)}\sum_1^n\E\rho_i(X_i,x_i)^t+C_B^{(t)}\Big(\sum_1^n\E\rho_i(X_i,y_i)^2\Big)^{t/2},   
\end{equation*}
where $C_A^{(t)}$ and $C_B^{(t)}$ are as in \eqref{eq:CA,CB} with $D=1$, 
\begin{gather*}
	C_t:=R(t,b_t),    
\end{gather*}	
and $b_t$ is the unique maximizer of 
$$R(t,b):=(b^{t - 1} + (1 - b)^{t - 1}) \big(b^{\frac1{t - 1}} + (1 - b)^{\frac1{t - 1}}\big)^{t - 1}$$  over all $b\in[0,\frac12]$. 
\end{corollary} 

This follows immediately from Corollary~\ref{cor:} and \cite[Corollary 3.1]{re-center}. 

An example of separately Lipschitz functions $g:\X^n\to\R$ is given by the formula 
\begin{equation*}
	g(x_1,\dots,x_n)=\|x_1+\dots+x_n\|
\end{equation*}
for all $x_1,\dots,x_n$ in any (not necessarily 2-smooth) separable Banach space $(\X,\|\cdot\|)$. 
In this case, one may take $\rho_i(\tilde x_i,x_i)\equiv\|\tilde x_i-x_i\|$. 
Thus, one immediately obtains 

\begin{corollary}\label{cor:conc-sums}
Let $X_1,\dots,X_n$ be independent random vectors in a separable Banach space $(\X,\|\cdot\|)$. 
Let here $Y:=\|X_1+\dots+X_n\|$. 
For each $i\in\intr1n$, take any $x_i$ and $y_i$ in $\X$.  
Then, with $C_A^{(t)}$, $C_B^{(t)}$, and $C_t$ as in Corollary~\ref{cor:concentr}, 
\begin{equation*}
	\E|Y-\E Y|^t
	\le C_t\, C_A^{(t)}\sum_1^n\E\|X_i-x_i\|^t+C_B^{(t)}\Big(\sum_1^n\E\|X_i-y_i\|^2\Big)^{t/2}.    
\end{equation*}
In particular, 
$C_3<1.316$ and hence, by \eqref{eq:t=3} with $D=1$, 
\begin{equation*}
	\E|Y-\E Y|^3
	\le1.316\,\sum_1^n\E\|X_i-x_i\|^3+2\Big(\sum_1^n\E\|X_i-y_i\|^2\Big)^{3/2};       
\end{equation*}
this improves the corresponding bound in \cite{re-center}, which had the constant factor $3$ in place of $2$. 
\end{corollary}


The separate-Lipschitz 
condition \eqref{eq:Lip} is easier to check than a joint-Lipschitz one. Also, the former is more generally applicable. On the other hand, when a joint-Lipschitz condition is satisfied, one can generally obtain better bounds. Literature on the concentration of measure phenomenon, almost all of it for joint-Lipschitz settings, is vast; let us mention here only \cite{ledoux-tala,ledoux_book,lat-olesz,
ledoux-olesz}. 

\section{Proofs}\label{proofs} 

The proof of Theorem~\ref{th:} is almost the same as that of \cite[Theorem~1]{pin80}; 
instead of \cite[Lemma~1]{pin80} one should now use the following more general lemma, valid for general 2-smooth Banach spaces. 

\begin{lemma}\label{lem:1}
Suppose that $t\ge2$ and $p(t)$ and $q(t)$ satisfy conditions \eqref{eq:p(s),q(s)}. 
Suppose also that the function $Q(\cdot):=\|\cdot\|^2$ is twice differentiable everywhere on $\X$ -- 
which may be assumed without loss of generality in view of \cite[Remark~2.4]{pin94}. 
Then for any $x$ and $y$ in $\X$ 
\begin{equation*}
	\|x+y\|^t-\|x\|^t\le\tfrac t2\|x\|^{t-2}Q'(x)(y)
	+\tfrac{t(t-2+D^2)}2\, p(t)\|x\|^{t-2}\|y\|^2+\tfrac{t-2+D^2}{t-1}q(t)\|y\|^t.   
\end{equation*}
\end{lemma}

\begin{proof}[Proof of Lemma~\ref{lem:1}]
Take indeed any $x$ and $y$ in $\X$ and introduce 
$$f(\th):=\|x_\th\|^t-\|x\|^t=Q(x_\th)-\|x\|^t,$$ 
where $\th\in[0,1]$ and $x_\th:=x+\th y$, so that $f(0)=0$, $f'(0)=\tfrac t2\|x\|^{t-2}Q'(x)(y)$, and $f(1)=\|x+y\|^t-\|x\|^t$. 
Moreover, since $\X$ is $(2,D)$-smooth and $Q$ is twice differentiable on $\X$, for all $\th$ one has $Q''(x_\th)(y,y)\le2D^2\|y\|^2$ and $|Q'(x_\th)(y)|\le2\|x_\th\| \|y\|$ (cf.\ \cite[Lemma~2.2 and Remark~2.4]{pin94}); so,  
\begin{align*}
	f''(\th)&=\tfrac t2\,(\tfrac t2-1)\, Q(x_\th)^{t/2-2} \big(Q'(x_\th)(y)\big)^2 
	+ \tfrac t2\, Q(x_\th)^{t/2-1} Q''(x_\th)(y,y) \\
	&\le t(t-2+D^2)\|x_\th\|^{t-2}\|y\|^2 \\  
	&\le t(t-2+D^2)\big(p(t)\|x\|^{t-2}\|y\|^2+q(t)\th^{t-2}\|y\|^t\big), 
\end{align*}
since $(\al+\be)^{t-2}\le p(t)\al^{t-2}+q(t)\be^{t-2}$ for all $\al$ and $\be$ in $[0,\infty)$. 
It remains to use the identities  $$\|x+y\|^t-\|x\|^t=f(1)=f(0)+f'(0)+\textstyle{\int_0^1}(1-\th)f''(\th)\dd\th.$$   
\end{proof}

\begin{proof}[Proof of Corollary~\ref{cor:}]
For any $j\in\intr0{m-1}$ and $\la_j>0$
\begin{multline}\label{eq:sum1}
	j!\,\sum_{J\in\J_{n,j}} \A{\mu_n(J)-1}{t-2j} \prod_{i\in J}b_i^2
	\le\A n{t-2j} B_n^{2j}
	\le(\A nt)^{\frac{t-2j-2}{t-2}} \, (\A n2 B_n^{t-2})^{\frac{2j}{t-2}} \\ 
	\le(\A nt)^{\frac{t-2j-2}{t-2}} \, (B_n^t)^{\frac{2j}{t-2}}
	\le\tfrac{t-2j-2}{t-2}\la_j^{-2j}\,\A nt + \tfrac{2j}{t-2}\,\la_j^{t-2j-2}\,B_n^t; 
\end{multline}
the second inequality here follows by the log-convexity of $\A ns$ in $s$, and the fourth one follows by Young's inequality $x^{1/p}y^{1/q}\le\frac xp+\frac yq$ for any $x$, $y$, $p$ and $q$ such that $x\ge0$, $y\ge0$, $p>0$, $q>0$, and $\frac1p+\frac1q=1$. 
Next, 
\begin{multline}\label{eq:sum2}
\sum_{J\in\J_{n,m}} \big(\A{\mu_n(J)-1}2\big)^{t/2-m} \prod_{i\in J}b_i^2
\le \sum_{J\in\J_{n,m}} B_{\mu_n(J)-1}^{t-2m} \prod_{i\in J}b_i^2 \\ 
=
\sum_{i_m\in\intr1n}b_{i_m}^2\sum_{i_{m-1}\in\intr1{i_m-1}}b_{i_{m-1}}^2\dots\sum_{i_1\in\intr1{i_2-1}} b_{i_1}^2 B_{i_1-1}^{t-2m} 
\le\Big(\prod_{j\in\intr1m}\frac1{t/2-m+j}\Big)B_n^t; 
\end{multline}
the iterated sum here is bounded by induction, using the inequality 
\begin{equation}\label{eq:sum-int}
	\sum_{i\in\intr1k} b_i^2 B_{i-1}^s\le\frac1{s/2+1}B_k^{s+2}
\end{equation}
for $s\ge0$ and $k\in\intr1\infty$; in turn, inequality \eqref{eq:sum-int} follows because its left-hand side is a left (and hence lower) Riemann sum for the integral $\int_0^{B_k^2}x^{s/2}\dd x$ of the nondecreasing function $\cdot^{s/2}$
, whereas the right-hand side of \eqref{eq:sum-int} is the value of the integral. 
Now Corollary~\ref{cor:} follows by \eqref{eq:}, \eqref{eq:sum1}, and \eqref{eq:sum2}. 
\end{proof}

\bibliographystyle{abbrv}


\bibliography{C:/Users/Iosif/Dropbox/mtu/bib_files/citations}

\end{document}